\documentclass[a4paper,12pt]{article}
\usepackage{amsmath,amssymb,amsthm}
\title{Qualitative Properties of Local Random Invariant
 Manifolds for SPDEs with Quadratic Nonlinearity }
\author{Dirk Bl\"omker\\
Institut f\"ur Mathematik\\
Universit\"at Augsburg\\
86135 Augsburg\\
dirk.bloemker@math.uni-augsburg.de
 \bigskip\\
 Wei Wang\\
School of Mathematics\\
 University of Adelaide\\
 5005 Adelaide Australia\\
 w.wang@adelaide.eu.au\\ \& \\
Department of Mathematics\\
Nanjing University\\
210093 Nanjing China\\
wangweinju@yahoo.com.cn}
\date{\today}

\newtheorem{theorem}{Theorem}
\newtheorem{lemma}{Lemma}
\newtheorem{definition}{Definition}
\newtheorem{corollary}{Corollary}
\newtheorem{remark}{Remark}
\newcommand{\eps}{\varepsilon}

\begin{document}
\maketitle

\begin{abstract}
The qualitative properties of local random invariant manifolds for
stochastic partial differential equations with quadratic
nonlinearities and multiplicative noise is studied by a cut off
technique. By a detail estimates on the Perron fixed point equation
describing the local random invariant manifold,
the structure near a bifurcation is given.
\end{abstract}

\section{Introduction}

Stochastic partial differential equations SPDEs arise as macroscopic
mathematical models of complex systems under random influences.
There have been rapid progresses in this area \cite{Gar, Roz, PZ92,
WaymireDuan, Bl:book, HuangYan}. More recently, SPDEs have been
investigated in the context of random dynamical systems (RDS)
\cite{Arn98}; see \cite[e.g.]{CLR00,CDSch03,Cheu95,CF94, CFD97,
Sch92, DLSch03, DLSch04}, among others.

Invariant manifolds are special invariant sets locally represented
by graphs in state spaces (function spaces) where the solution
process of the system is defined. A random invariant manifold
provides a geometric structure to reduce stochastic dynamics. In
fact the dynamics of the system is completely determined by that on
the globally attracting invariant manifold.
Nevertheless, due to the inherit non-autonomous nature of the SPDE,
these manifolds move in time. So one can reduce the
system to a lower dimensional system on the invariant manifold
together with the dynamics of that manifold.

There are some
results on the existence of random invariant manifold for a class of
stochastic partial differential equations with Lipschitz
nonlinearity, \cite{DLSch03, DLSch04, LSch07} etc.\
In \cite{CLR00} an estimate of dimension of the local invariant
manifold  is given, in order
to study a pitchfork type bifurcation for stochastic
reaction-diffusion equations with cubic nonlinearity.

In the recent work \cite{WD07}, by a detail study of the properties
of flow, an invariant manifold reduction of a class stochastic
partial differential equations with Lipschitz nonlinearity is
obtained. The reduced equation is a stochastic differential equation
defined on a finite dimensional invariant manifold.
Moreover in \cite{Duan:Preprint} the expansion in leading order terms
of a random invariant manifold for noise-strength to $0$ is considered.

There is a difficulty in reducing  infinite dimensional stochastic
systems with non-Lipschitz nonlinearity to a lower dimensional
system on an invariant manifold. But it can be done near an
equilibrium solution by introducing a cut-off function near the
equilibrium solution such that the nonlinearity becomes globally
Lipschitz in an approximate space. Similar ideas are already
used in \cite{WD07} and \cite{CLR00}.
In this situation for the equation with cut-off
the global approach in \cite{WD07}
could be followed,
and the result for the original equation is then a local one.

However,  let us point out that the
invariant manifold is still moving  in the infinite dimensional
space.  It is in general difficult  to see clearly  the dynamics of
the reduced system even though it is essentially finite dimensional.
Therefore we aim at the local shape of the invariant manifold.

At the same time an  amplitude equation is an important tool to
describe qualitatively the dynamics of stochastic systems near a
change stability, which has been studied heavily for additive noise
(\cite[e.g.]{BL03,BH04,BHP07}) and for multiplicative noise (see
\cite{Bl:book, BHP07b}). In contrast to random invariant manifolds,
which move in time, the dynamics studied via amplitude equations
approximates the dynamics on a given fixed vector space, and the
essential dynamics is given by a stochastic ordinary differential
equation on dominant modes. The drawback in this case is that the
results
 hold only with high probability.
Results that hold almost surely are not possible to obtain.

 In this paper we  study the properties of the random invariant manifold locally.
In fact we consider properties of local invariant manifolds for
equations  near a change of stability with quadratic nonlinearities
and multiplicative noise. One of the simplest examples is a
Burgers-type equation
\begin{equation}
\label{e:burger}
\partial_t u =\partial_x^2 u+u  +\nu u + \tfrac12 \partial_x u^2 + \sigma u \circ \dot{W}
\end{equation}
subject to Dirichlet boundary conditions on $[0,\pi]$ with
one-dimensional multiplicative noise of Stratonovic-type.

There are many examples of SPDEs near its first
change of stability, where our results apply.
For instance, we can consider
the Kuramoto-Sivashinsky equation, a model from surface
growth (cf. \cite{Bl:surf}),
Navier-Stokes equations with an additional linear term,
or the Rayleigh-Benard system near the convective instability.

For our main results we reduce the cut-off equation of (\ref{e:burger}) to a
stochastic differential equation on an attracting random invariant manifold in a
small cut-off ball. Our results extend \cite{CLR01}, where the local
dimension of a random invariant manifold near a fixed point  is
studied. Using the linearization in our case the local dimension
is one, but here we show that with high probability
the manifold is locally the graph
of a quadratic function over a one-dimensional
space (cf. Theorem \ref{thm:shape}).

Let us compare our results with amplitude equations.
Typically, one assumes there the scaling
$\sigma=\varepsilon$, $\nu=\nu_0\varepsilon^2$ for
some small $\varepsilon$  with $|\nu_0|\leq 1$.
In that scaling our Burgers-type equation reads
\begin{equation}\label{e:re-burger}
\partial_t u =\partial_x^2 u+u  +\nu_0\varepsilon^2 u + \tfrac12 \partial_x u^2 + \varepsilon u \circ
\dot{W}\,.
\end{equation}
In \cite{Bl:book} it was shown that for solutions of size ${\cal
O}(\varepsilon)$ with  probability  higher than $1-C_p\varepsilon^p$ for
all $p>1$
$$
u(t,x)\approx \varepsilon a(\varepsilon^2t)\sin
$$
on intervals of length ${\cal O}(\varepsilon^{-2})$, where
$$
\partial_T a = \nu_0 a - \tfrac1{12}a^3 + a\circ\dot{W}
$$
Our main results of this paper show for $\nu$ and $\sigma$ small,
but without scaling assumption, that there is a time dependent
random Lipshitz map $\phi(t,\omega,\cdot ):\mathbb{R}\to\mathbb{R}$
describing locally the random invariant manifold
such that near $0$ the flow along the manifold is
\begin{equation}
\partial_t a=\nu a+ a\phi(t,a)+\sigma a\circ \dot{W}.
\end{equation}
Moreover, for a given $t$ the probability that
$$\phi(t,a)\approx -\tfrac1{12}a^2
$$
is larger than  $1-C\exp\{-1/\eps\}$.

The rest of the paper is organized  as follows. In Section
\ref{sec:result}, assumptions and main results are presented.
Section \ref{sec:exLRIM} states the results on existence of
random invariant manifolds using a cut-off technique.
While Section \ref{sec:shapeIM} provides results on the
local structure of the manifold.


\section{Setting and Results}\label{sec:result}


Consider the following abstract equation
\begin{equation}
\label{e:main}
\partial_t u = -Lu +\nu u +B(u,u) + \sigma u \circ \dot{W},\;\;
u(0)=u_0\in H
\end{equation}
where $W$ is a standard real valued Brownian motion, and the noise
is in the Stratonovic sense. The real constants $\nu$ and $\sigma$
describe the distance from the bifurcation and the noise strength,
respectively. Assume $H$ is a real separable Hilbert space with norm
$\|\cdot\|$ and scalar product $\langle\cdot,\cdot\rangle$.

 We make the following assumptions.
\begin{flushleft}
{\bf Assumption A$_1$} {\it Let L be a non-negative self-adjoint
unbounded operator with compact resolvent  in the Hilbert space H,
and suppose that $\{e_k\}_{k=1}^\infty$ and
$\{\lambda_k\}_{k=1}^\infty$ is an orthonormal basis of
eigenfunctions and the corresponding ordered eigenvalues
($\lambda_{k+1} \ge \lambda_k$).}
\end{flushleft}

\begin{flushleft}
{\bf Assumption A$_2$} {\it The kernel of $L$, $H_c:={\rm
span}\{e_1,\ldots,e_N\}$, is finite dimensional, i.e. $\lambda_1=\ldots=\lambda_N=0$ and
$\lambda_*=\lambda_{N+1}>0$. }
\end{flushleft}
Denote by $P_c$ the orthogonal projection from $H$ to $H_c$.
Furthermore, $P_s=I-P_c$ and let $L_s=P_s L$\,. In the following
we use the subscript 'c' always for projection onto $H_c$ and
 's'  for projection onto $H_s$.
\begin{remark}
Note that the assumption that the  eigenfunctions form an
orthonormal set is for simplicity. We could relax this, assuming
that $P_c$ is a continuous projection commuting with $L$,
which is no longer self-adjoint in that case.
\end{remark}
Let us now  introduce the interpolation spaces $H^\alpha$,
$\alpha>0$, as the domain of $L^{\alpha/2}$ endowed with scalar
product $\langle u, v \rangle_\alpha=\langle u, (1+L)^\alpha
v\rangle$ and corresponding norm $\|\cdot\|_\alpha$. Furthermore, we
identify $H^{-\alpha}$ as the dual of $H^\alpha$ with respect to  the
inner product in $H$.

We make the following assumption on the nonlinearity.

\begin{flushleft}
{\bf Assumption A$_3$}

{\it For some $\alpha\in (0, 1)$, let $B : H\times H \rightarrow
H^{-\alpha}$ be a bounded bilinear and symmetric  operator, i.e.
$B(u,\bar{u})=B(\bar{u},u)$ and  there is a constant~$C_B>0$ such
that $\|B(u,\bar{u}) \|_{-\alpha} \le C_B  \|u \| \|\bar{u} \|$.
Also denote by $\langle\cdot,\cdot\rangle$ the dual paring between
$H^{-\alpha}$ and $H^\alpha$. Then suppose that for $B(u)\triangleq
B(u,u)$
$$\langle B(u),u \rangle = 0
\quad \text{ for all }u \in H^\alpha\,.
$$}
\end{flushleft}
In the following we denote $P_cB(\cdot, \cdot) $ by $B_c(\cdot,
\cdot)$ and $P_sB(\cdot, \cdot) $ by $B_s(\cdot, \cdot)$\,. From
bilinearity and boundedness, we immediately obtain a local Lipschitz
condition for $B$. To be more precise,  for all $R>0$
\begin{equation}
\|B(u)-B(\bar{u}) \|_{-\alpha}
=\|B(u - \bar{u}, u + \bar{u}) \|_{-\alpha}
\leq 2RC_B \|u-\bar{u}\|
\end{equation}
for all $u$, $\bar{u}\in H$ with $\|u\|\leq R$ and $\|\bar{u}\|\leq
R$\,.

Let us remark that the symmetry of $B$ is not necessary,
but it simplifies for instance expansions of $B(u+v)$.
Moreover, one can without loss of generality always assume
that the quadratic form is given by a symmetric $B$.
If not one can consider $\tilde{B}(u,v)=\tfrac12(B(u,v)+B(v,u))$.


\subsection*{Random Dynamical Systems}

Before giving our main result, we recall some basic theory of random
dynamical systems. We will work on the canonical probability space
$(\Omega_0, \mathcal{F}_0, \mathbb{P})$ where the sample space
$\Omega_0$ consists of the sample paths of $W(t)$. To be more precise
$W$ is the identity on $\Omega_0$, with
$$
\Omega_0=\{w\in C([0, \infty), \mathbb{R}): w(0)=0 \}\,,
$$
and $\mathbb{P}$ the Wiener measure.
For more details see \cite{Arn98}.

 Let $\theta_t: (\Omega_0, \mathcal{F}_0,
\mathbb{P})\rightarrow (\Omega_0, \mathcal{F}_0, \mathbb{P}) $ be a
metric dynamical system (driven system), that is,
\begin{itemize}
    \item $\theta_0=id$,
    \item $\theta_t\theta_s=\theta_{t+s}$,\;\; for all $s$,
        $t\in\mathbb{R}$,
    \item the map $(t,\omega)\mapsto \theta_t\omega$ is
     measurable and $\theta_t\mathbb{P}=\mathbb{P}$ for all
        $t\in\mathbb{R}$.
\end{itemize}
On $\Omega_0$ the map  $\theta_t$ is  the shift
\begin{equation}\label{MD}
\theta_t\omega(\cdot)=\omega(\cdot+t)-\omega(t)\,,\;\; t\in
\mathbb{R}\,, \omega\in\Omega_0\,.
\end{equation}
\begin{definition}
\label{def:RDS}
Let $(\mathcal{X}, d)$ be a metric space with Borel $\sigma$-algebra
$\mathcal{B}$, then a random dynamical system on $(\mathcal{X}, d)$ over
$\theta_t$ on {\rm(}$\Omega_0, \mathcal{F}_0, \mathbb{P}${\rm )} is
a measurable map
\begin{eqnarray*}
\varphi:\mathbb{R}^+\times \Omega_0\times \mathcal{X}&\rightarrow&
\mathcal{X}\\
(t, \omega, x) &\mapsto& \varphi(t, \omega)x
 \end{eqnarray*}
  having the following cocycle property
  $$
\varphi(0, \omega) x=x,
\quad
\varphi(t, \theta_\tau\omega)\circ\varphi(\tau, \omega) x
=\varphi(t+\tau,\omega)x
 $$
 for $t, \tau\in\mathbb{R}^+$, $x\in \mathcal{X}$ and $\omega\in\Omega$.
\end{definition}
A RDS $\varphi$ is continuous or differentiable if $\varphi(t,
\omega) : \mathcal{X}\rightarrow \mathcal{X}$ is continuous or
differentiable (see \cite{Arn98} for more details on RDS).

It is well known that in order to show that (\ref{e:main}) generates
an RDS, one can rely on a random transformation to a random PDE. For
this we introduce the following real-valued stationary  process.
\begin{definition}
Define
\begin{equation}
z(t)= z(\theta_t\omega)
\end{equation}
on $\Omega_0$\,, where
\begin{equation}
\label{e:defz}
 z(\omega)=\sigma\int_{-\infty}^0 e^s \omega(s)ds\,.
\end{equation}
\end{definition}
Now the mapping $t\mapsto  z(\theta_t\omega)$ is continuous and solves
\begin{equation}\label{e:ou}
dz+zdt=\sigma d\omega
 \quad \text{or}  \quad
z(\theta_t\omega)=z(\omega) - \int_0^t z(\theta_s\omega)ds+\sigma \omega(t).
\end{equation}
Moreover,
\begin{equation*}
\lim_{t\rightarrow\pm\infty}\frac{|z(\theta_t\omega)|}{|t|}=0\;\;{\rm
and }\;\; \lim_{t\rightarrow\pm\infty}\frac{1}{t}\int_0^t
z(\theta_\tau\omega)d\tau=0 \text{\  for a.e.\ }\omega\in\Omega_0\,.
\end{equation*}
The above properties hold in a $\theta_t$ invariant set
$\Omega\subset\Omega_0$ of full probability, see \cite{DLSch03}.
In the following,
we frequently use the shorthand notation $z(t)=z(\theta_t\omega)$.

By the transformation
\begin{equation}\label{e:trans}
v(t)=e^{-z(t)}u(t)
\end{equation}
Equation (\ref{e:main}) becomes
\begin{equation}\label{e:v-main}
\partial_t v=-Lv+zv+\nu v+e^{z}B(v,v),
\end{equation}
which is an evolutionary equation with random stationary
coefficient. Then for almost all $\omega\in\Omega$\,, by the same
discussion for the wellposedness to deterministic evolutionary
equation \cite{Tem97}, for any $t_0<T\in\mathbb{R}$ and $v_0\in H$
there is a unique solution $v(t,\omega; t_0,v_0)\in C(t_0,T;H)$ of
equation (\ref{e:v-main}) with $v(t_0)=v_0$ and the map $v_0\mapsto
v(t,\omega; t_0, v_0)$ is continuous for all $t\geq t_0$\,. Then by
the stationary transformation (\ref{e:trans}) we have the following
result.
\begin{theorem}
\label{thm:exRDS} Under Assumptions $A_1-A_3$,  (\ref{e:main})
generates a continuous RDS $\varphi(t, \omega)$ on $H$.
\end{theorem}
For a continuous random dynamical system $\varphi$ on $\mathcal{X}$
given by Definition \ref{def:RDS}, we  need the following  notions
to describe its dynamical behavior. First for any subset
$M\subset\mathcal{X}$ define $d(M)=\sup_{x\in M}d(x,0)$ with
$0\in\mathcal{X}$ is the zero element in $\mathcal{X}$ and define
$B_R(0)=\{x\in\mathcal{X}: d(x,0)<R\}$\,.
\begin{definition}
A collection $M=M(\omega)_{\omega\in\Omega}$ of non-empty closed
sets $M(\omega)\subset \mathcal{X}$, $\omega\in\Omega$  is a random
set, if
\begin{equation*}
\omega\mapsto\inf_{y\in M(\omega)}d(x, y)
\end{equation*}
is a real valued random variable for any $x\in \mathcal{X}$.
\end{definition}
\begin{definition}
A random set $M(\omega)$ is called a tempered absorbing set for a
random dynamical system $\varphi$ if for any bounded set
$K\subset\mathcal{X}$ there exists $t_K(\omega)$ such that $\forall
t\geq t_K(\omega)$
\begin{equation*}
\varphi\big(t,\omega\big)K\subset M(\theta_t\omega).
\end{equation*}
and for all $\varepsilon>0$
\begin{eqnarray*}
 \lim_{t\rightarrow\infty}e^{-\varepsilon
t}d\big(M(\theta_t\omega)\big)=0, \ \text{for a.e.}\ \omega\in
\Omega\,.
\end{eqnarray*}

\end{definition}
For more details about random set we refer to \cite{CF94}.
\begin{definition}
A random set $M(\omega)$ is called a positive invariant set for a
random dynamical system $\varphi(t,\omega)$ if
\begin{equation*}
\varphi(t,\omega) M(\omega)\subset M(\theta_t\omega),\;\;for\;t\geq
0.
\end{equation*}
If $M(\omega)=\{x_1+\psi(\omega, x_1)| x_1\in \mathcal{X}_1 \}$ is
the graph of a random Lipschitz mapping
\begin{equation*}
\psi(\omega, \cdot): \mathcal{X}_1\rightarrow \mathcal{X}_2
\end{equation*}
with $\mathcal{X}=\mathcal{X}_1\oplus \mathcal{X}_2$,
then $M(\omega)$ is called a Lipschitz invariant manifold of
$\varphi$.
\end{definition}
For further details about the random invariant manifold theory,
see \cite{DLSch03}.
\begin{definition} \label{def:LRIM}
System (\ref{e:main}) is said to have a local random invariant
manifold (LRIM) with radius $R$, if there is a random set
$\mathcal{M}^R(\omega)$, which is defined by the graph of a random
continuous function $\psi(\omega,\cdot ):
\overline{B_R(0)}\cap\mathcal{X}_1\rightarrow \mathcal{X}_2$\,, such
that for all bounded sets $B$ in $B_R(0)\subset\mathcal{X}_1$ we have
$$
 \varphi(t, \omega)[\mathcal{M}^R(\omega)\cap B]\subset \mathcal{M}^R(\theta_t\omega)
$$
 for all $t\in(0,\tau_0(\omega))$ with
 \begin{equation}\label{e:tau0}
 \tau_0(\omega)=\tau_0(\omega,B) \\
 =\inf\{t\ge0:\varphi(t, \omega)[\mathcal{M}^R(\omega)\cap B]  \not\subset B_R(0)\}.
 \end{equation}
\end{definition}

In \cite{LSch07} the existence of a random
invariant set on a tempered ball around $0$
was established. This should be possible,
if we have a LRIM. But we will not focus on that.
In the following we are aiming on local properties
of the LRIM.

Due to the method of proof, where we rely on a cut-off
at radius $R$, we always obtain a globally defined
Lipschitz invariant manifold $\mathcal{M}(\omega)$
for the system with cut-off. Now, as both flows agree
on $B_R(0)$ it is easy to check that
$\mathcal{M}^R(\omega)=\mathcal{M}(\omega)\cap B_R(0)$
defines a LRIM. In the following, we will mainly work
with the globally defined $\mathcal{M}(\omega)$.

\subsection*{Main Results}

We prove the following theorem in Section \ref{sec:exLRIM}.
See Theorems \ref{thm:fixpoint} and \ref{thm:reduce}.
This is based on the properties for the system with cut-off.
\begin{theorem} {\bf (Existence)} \label{thm:exRIM}
Under Assumptions $A_1-A_3$, the random dynamical system $\varphi(t, \omega)$ defined by
(\ref{e:main}) has a LRIM $\mathcal{M}^R(\omega)$ for sufficiently small $R>0$.
This manifold is given as the graph of
a random Lipschitz map $h(\omega,\cdot):H_c\to H_s$:
 $$
\mathcal{M}^R(\omega)=\{(\xi,e^{z(\omega)}h(\omega,e^{-z(\omega)}\xi)):\ \xi \in H_c\}\;.
 $$
Moreover, if $\lambda_*>4\nu$, then the LRIM $\mathcal{M}^R(\omega)$
is locally exponentially attracting almost surely in the small ball
$B_R(0)$. That is for any $\|u_0\|<R$
 $$
dist\big( \varphi(t, \omega) u_0,
\mathcal{M}^R(\theta_t\omega)\big)\leq 2RD(t, \omega)e^{-\lambda_*t}
 $$
for all $t<\tau_0(\omega)=\inf\{t>0: \varphi(t, \omega) u_0\not\in B_r(0)\}$
with $D(t, \omega)$ is a tempered increasing process, see (\ref{e:D}).
 \end{theorem}
We are interested in the property of solutions on the LRIM
$\mathcal{M}^R(\omega)$ and the
dynamics of $\mathcal{M}^R(\omega)$ itself
 near the first bifurcation $\nu=0$.

It is proved that

\begin{theorem} {\bf (Local Shape)} \label{thm:shape}
Let Assumptions $A_1-A_3$ be true. Suppose $\sigma>0$,
$|\nu|<\sigma$ and $R\le1$, and let $h$ be the fixed point given by
Theorem \ref{thm:fixpoint}. Then for $\sigma \to 0$,
\begin{equation}\label{e:claim_shape}
 \|e^{z(\omega)}h(\omega, e^{-z(\omega)}\xi) -  L_s^{-1}B_s(\xi,\xi)\|
\le C(\|\xi\|+R^2+\sqrt\sigma )\cdot\|\xi\|^2\;,
\end{equation}
holds with probability larger than $1-C\exp\{-1/\sqrt{\sigma}\}$ for
all $\|\xi\| \le \tfrac12 R$.
\end{theorem}
\begin{remark}
The previous result will not apply to deterministic equations,
as we always assume $|\nu|<\sigma$, but it allows for large
radii $R<1$ for the cut-off.
\end{remark}

Let us give a more general example than the Burgers'
equation already mentioned.
If we consider the space $H_c$ being one-dimensional,
then we can write $\xi=\alpha\cdot e$ for a basis function $e\in H_c$.

Thus the invariant manifold $\mathcal{M}^R(\omega)$
is locally given (with high probability as the graph of
$$L_s^{-1}B_s(\xi,\xi) = \alpha^2 L_s^{-1}B_s(e,e) =: \alpha^2 v_s \in H_s.$$
Now $\mathcal{M}^R(\omega)$ lies approximately in the
plane spanned by $e$ and $v_s$ and is given by a parabola,
unless $v_s=0$. In that case the manifold is approximately flat.

Let us finally comment on the flow on the manifold. Using Theorem
\ref{thm:exRIM}, it is easy to describe the flow along the invariant
manifold $\mathcal{M}(\omega)$ on a small ball of radius $R/2$
around $0$ by the following equation
$$
\partial_t u_c
=\nu u_c+P_cB^{(R)}(u_c+e^{z(\theta_t\omega)}h(\theta_t\omega,
e^{-z(\theta_t\omega)}u_c))+ \sigma u_c\circ \dot{W}(t).
$$
By Theorem \ref{thm:shape} we can replace
$e^{z(\theta_t\omega)}h(\theta_t\omega, e^{-z(\theta_t\omega)}u_c)$
by
\begin{equation}
\tag{*} L_s^{-1}B_s(u_c, u_c)+r(t)\,,
\end{equation}
where the probability of $r(t)$ being large is bounded by
$C\exp\{-1/\sqrt{\sigma}\}$.

This finally yields  an  equation, where
$$
\partial_t u_c=\nu u_c+P_cB^{(R)}(u_c+L_s^{-1}B_s(u_c,u_c)))+
 \sigma u_c\circ \dot{W}(t).
$$
This rederives an amplitude equation for the equation. Nevertheless,
for a detailed analysis of the flow on $\mathcal{M}$, we would need
bound on $r$, which are uniform in time. This will be postponed to
future work.

To conclude the presentation of the main results,
we remark that it should be straightforward to generalize
the presented results to higher, but finite dimensional noise.


\section{Existence of Local Invariant Manifold}
\label{sec:exLRIM}


In this section we construct a LRIM for system (\ref{e:main})
for small $R>0$ and prove its exponentially attracting property.
For this we rely on a cut-off technique given by the following definition.
As explained after Definition \ref{def:LRIM} (cf. also \cite{CLR00,LSch07}),
the LRIM is given by the RIM for the cut-off system.

\begin{definition}\label{def:cutoff}{\bf (cut-off)}
Let $\chi:H\rightarrow \mathbb{R}$ be a bounded smooth function such
that $\chi(u)=1$ if $\|u\|\leq 1$ and $\chi(u)=0$ if $\|u\|\geq 2$.
For any  $R>0$, we define $\chi_R(u)=\chi(u/R)$ for all $u\in H$.

Given a radius $R>0$ we define
$$
B^{(R)}(u)= \chi_R(u) B(u, u).
$$
\end{definition}
Now by Assumption (A$_3$) for given $1>\alpha>0$, the operator
$B^{(R)}$ is globally Lipschitz-continuous from space $H$ to $H^{-\alpha}$
with Lipschitz  constant
\begin{equation}
 \label{e:Lipcut}
\text{Lip}_{H, H^{-\alpha}}(B^{(R)})= L_R:=2RC_B\,.
\end{equation}
Consider now  the following cut-off system
\begin{equation}
\label{e:cut-off}
\partial_t u = -Lu +\nu u +B^{(R)}(u) + \sigma u \circ \dot{W}\,,\qquad
 u(0)=u_0\,.
\end{equation}
As before by the transformation $v=ue^{-z}$
(with $z=z(t)=z(\theta_t\omega)$),
we have
\begin{equation}\label{e:rand eq}
v_t=-L v+z v+\nu v+e^{-z}B^{(R)}(e^z v) \,,\qquad
 v(0)=u_0e^{-z(0)}\,.
\end{equation}
In order to obtain a random invariant manifold for random dynamical
system $\varphi^R(t, \omega)$ defined by the above stochastic
equation (\ref{e:cut-off}) we study the transformed equation
(\ref{e:rand eq}). By the projections $P_c$ and $P_s$ equation
(\ref{e:rand eq}) is split into
\begin{eqnarray*}
\partial_tv_c =&&
\nu v_c+zv_c+P_ce^{-z}B^{(R)}(e^z v),\qquad
v_c(0)=P_c u_0e^{-z(0)},\\
\partial_tv_s =&
-L_sv_s+&\nu v_s+zv_s+P_se^{-z}B^{(R)}(e^z v),\qquad
v_s(0)=P_su_0e^{-z(0)}.
\end{eqnarray*}
We use the Lyapunov--Perron method on  the following random space
with random norm depending on $\omega\in\Omega$. The process $z$ was
defined in (\ref{e:ou}).
\begin{definition}\label{def:Ceta-space}
For  $-\nu<\eta< \lambda_*-\nu$  define the Banach space
\begin{equation}\label{def:space}
C^-_\eta=\Big\{v\in C((-\infty, 0], H): \sup_{t\le0}
\{e^{\eta t-\int_0^t z(s)ds}\|v(t)\|\} <\infty\Big\}
\end{equation}
with norm
$$
\|v\|_{C^-_\eta}=\sup_{t\le0}\{e^{\eta
t-\int_0^tz(\tau)d\tau}\|v(t)\|\} <\infty.
$$
\end{definition}
\begin{definition}\label{def:operator}
Define  the nonlinear operator $\mathcal{T}$ on $C^-_\eta$ for  given $\xi\in H_c$
and $\omega\in \Omega_0$ as
\begin{eqnarray}
\mathcal{T}(v, \xi)(t)&=&
e^{\nu t+\int_0^t z(s)ds}\xi
+ \int_0^te^{\nu (t-\tau)+ \int_\tau^t z(r)dr}
e^{-z(\tau)}P_cB^{(R)}\left(v(\tau)e^{z(\tau)}\right) d\tau \nonumber\\
&&+\int^t_{-\infty}e^{(-L_s+\nu)(t-\tau)
+\int_\tau^tz(r)dr}e^{-z(\tau)}P_sB^{(R)}\left((v(\tau)e^{z(\tau)}\right)d\tau
\,. \label{def:T}
\end{eqnarray}
\end{definition}
 By the Lipschitz property of $B^{(R)}$ (cf. (\ref{e:Lipcut}))
 it can be verified directly that for any $\xi\in H_c$ and
 $\omega\in\Omega_0$\,,
 $\mathcal{T}(\cdot, \xi):C_\eta^-\rightarrow
C_\eta^-$. Further, a short calculation shows that it is a Lipschitz
continuous map. To be more precise, for any $v$\,, $\bar{v}\in
C_\eta^-$
\begin{eqnarray*}
\lefteqn{\big[\mathcal{T}(v)(t)-\mathcal{T}(\bar{v})(t)\big] e^{-\int_0^t z(r)dr}}\\
&=&
+ \int_0^te^{\nu (t-\tau)+ \int_0^\tau z(r)dr-z(\tau)}
\left[ P_cB^{(R)}\left(v(\tau)e^{z(\tau)}\right)
    - P_cB^{(R)}\left(\bar{v}(\tau)e^{z(\tau)}\right)\right] d\tau\\
&&+\int^t_{-\infty}e^{(-L_s+\nu)(t-\tau)-\int_0^\tau z(r)dr-z(\tau)}
\left[ P_sB^{(R)}\left((v(\tau)e^{z(\tau)}\right)
    -P_sB^{(R)}\left((\bar{v}(\tau)e^{z(\tau)}\right)\right]d\tau\,.
\end{eqnarray*}
Now we use that for all $t >0$ and all $u\in H_s$
\begin{equation}\label{e:SG}
 \|e^{Lt} u\| \le M_{\alpha,\lambda} e^{- t\lambda}t^{-\alpha} \|u\|_{-\alpha}
\end{equation}
for some $\lambda<\lambda_*$ sufficiently close,
together with (\ref{e:Lipcut}) and the estimate
$\|P_c v\| \le C_\alpha \|P_c\|_{-\alpha}$
for some constant $C_\alpha$.

We obtain for $0<\eta+\nu<\lambda$
\begin{eqnarray*}
\|\mathcal{T}(v)-\mathcal{T}(\bar{v})\|_{C^-_\eta}
&\leq &
\int_0^te^{(\nu+\eta) (t-\tau)} C_\alpha L_R
    \|v-\bar{v}\|_{C^-_\eta} d\tau\\
&&  +\int^t_{-\infty}e^{(-\lambda+\nu+\eta)(t-\tau)} M_{\alpha,\lambda}L_R
    \|v-\bar{v}\|_{C^-_\eta}d\tau \\
&\leq &
L_R \Big[\frac{C_\alpha}{\eta+\nu}
+ M_{\alpha,\lambda}\frac{\Gamma(1-\alpha)}{(\lambda-\eta-\nu)^{1-\alpha}}\Big]
 \|v-\bar{v}\|_{C^-_\eta} \; .
\end{eqnarray*}
Note that our bound on $\text{Lip}(\mathcal{T})$
is actually  independent of $\xi$ and $\omega$.

Now for all $\xi \in H_c$ and $\omega\in \Omega_0$ the operator
$\mathcal{T}: C_\eta^-\rightarrow C_\eta^-$ is a contraction
provided $R$ is sufficiently small such that
\begin{equation}\label{e:condition1}
L_R \Big[\frac{C_\alpha}{\eta+\nu}
+ M_{\alpha,\lambda}\frac{\Gamma(1-\alpha)}{(\lambda-\eta-\nu)^{1-\alpha}}\Big]< 1.
\end{equation}
The celebrated theorem of  Banach  yields the following theorem.
\begin{theorem} \label{thm:fixpoint}
Suppose $0<\eta+\nu<\lambda_*$ and (\ref{e:condition1}).
Then the  operator  $\mathcal{T}$ has a
unique fixed point $v^*=v^*(\omega,\xi)\in C_\eta^-$.
\end{theorem}
Define  $h(\omega,\xi)=P_s v^*(0,\omega; \xi)$. Thus, by Definition \ref{def:operator}
 \begin{equation}\label{e:h}
h(\omega, \xi)=\int^0_{-\infty} e^{(L_s-\nu)\tau+\int_\tau^0z(r)dr}
e^{-z(\tau)}P_sB^{(R)}(v^*(\tau,\xi)e^{z(\tau)})d\tau.
 \end{equation}
Then by the same discussion as in \cite{DLSch04},
\begin{equation}\label{e:RIM}
\mathcal{M}_{cut}^R(\omega)=\{(\xi, e^{z(\omega)}h(\omega,
e^{-z(\omega)}\xi) ): \xi\in H_c \}
 \end{equation}
is a random invariant manifold for the random dynamical system
$\varphi^R(t, \omega)$,
 which is the graph of $e^{z(\omega)}h(\omega,e^{-z(\omega)}\xi)$\,.
Now define a Lipschitz mapping~$\psi$~by
\begin{eqnarray*}
\psi(\omega, \cdot): H_c\cap B_R(0)&\rightarrow& H_s\,,\\
\xi&\rightarrow&\psi(\omega,\xi)=e^{z(\omega)}h(\omega,
e^{-z(\omega)}\xi)\,.
\end{eqnarray*}
Then, as already indicated after the definition, it
is easy to check, that
\begin{equation}\label{e:LRIM}
 \mathcal{M}^R(\omega)=\text{graph}(\psi(\omega,\cdot))=\mathcal{M}^R_{cut}(\omega)\cap B_R(0)
 \end{equation}
defines a LRIM of the random dynamical system $\varphi(t,\omega)$.

Now we prove the attracting property of the random invariant
manifold $\mathcal{M}_{cut}^R(\omega)$ for equation
(\ref{e:cut-off}). We follow the approach of \cite{WD07}. It is
enough to prove the cone invariance property:
\begin{lemma}
 \label{lem:cone}
Fix $\delta>0$ and  define the cone
$$
{\cal K}_\delta = \{ u\in H\ :\ \|u_s\|  < \delta\|u_c\|  \|\}\;.
$$
Suppose that $R$ is sufficiently small such that
\begin{equation}\label{e:condition2}
\lambda_* \geq 2(1+\tfrac1\delta)^2L_R^2 + 4(1+\delta)L_R
\end{equation}
and
\begin{equation}\label{e:condition3}
\lambda_* > 4\nu  +2 L_R^2(1+\tfrac1\delta)^2\;.
\end{equation}
Let $v$, $\bar{v}$ be two solutions of (\ref{e:rand eq}) with
initial value $v_0=u_0e^{z(\omega)}$ and $\bar v_0=\bar
u_0e^{z(\omega)}$.

If $v(t_0)-\bar{v}(t_0)\in {\cal K}_\delta$, then
$v_c(t)-\bar{v}_c(t)\in {\cal K}_\delta$ for all $t\ge t_0$\,.

Moreover, if $v-\bar{v}$ is outside ${\cal K}_\delta$ at some time
$t_0$ then
\begin{eqnarray}
\label{e:bou_q}
\lefteqn{ \|v_s(t,\omega)-\bar{v}_s(t,\omega)\|^2}\\
&\le& \|u_0-\bar{u}_0\|^2
 \exp\left\{-\tfrac12\lambda_*t+z(\omega)+2\int_0^t z(\theta_\tau\omega)d\tau \right\}
\nonumber
\end{eqnarray}
 for  all $t\in[0,t_0]$.\\
\end{lemma}
In the proof we will see that
 the $4\nu$ is  not optimal in (\ref{e:condition3}),
but for simplicity of proof, we keep the $4\nu$.
\begin{proof}
Define
$$ p=v_c-\bar{v}_c
 \quad\mathrm{and}\quad
q=v_s-\bar{v}_s\,,
$$
then
$$
\partial_t p= \nu p+zp+e^{-z}P_cB^{(R)}(ve^z)-e^{-z}P_cB^{(R)}(\bar{v}e^z)
$$
$$
\partial_t q= -L_sq+\nu q+zq+e^{-z}P_sB^{(R)}(ve^z)-e^{-z}P_sB^{(R)}(\bar{v}e^z)\,.
$$
By the property of operator $L$ and the Lipschitz property of $B^{(R)}$ we
obtain for some positive constant $c_1$ depending on $\lambda_*$ and
$\alpha$
\begin{equation}\label{e:p}
\frac{1}{2}\frac{d}{dt}\|p\|^2
\geq \nu\|p\|^2+z\|p\|^2- L_R\|p\|^2-L_R\|p\|\|q\|
 \end{equation}
and
\begin{equation}\label{e:q}
\frac{1}{2}\frac{d}{dt}\|q\|^2\leq - \|q\|_1^2+\nu\|q\|^2+z\|q\|^2
+L_R\|q\|\|q\|_\alpha+ L_R\|p\|\|q\|_\alpha\;.
\end{equation}
Thus
 \begin{eqnarray*}
\lefteqn{\frac{1}{2}\frac{d}{dt}(\|q\|^2-\delta^2\|p\|^2)} \\
&\leq& - \|q\|_1^2+\nu\|q\|^2+z\|q\|^2+ L_R\|q\|\|q\|_\alpha+L_R\|p\|\|q\|_\alpha \\
&&-\nu\delta^2\|p\|^2-z\delta^2\|p\|^2+\delta^2 L_R\|p\|^2+\delta^2 L_R\|p\|\|q\|
\;. \\
\end{eqnarray*}
Now suppose that $v-\bar{v}\in {\cal K}_\delta $ (i.e. $\delta \|p\|=\|q\|$) for some $t$.
Then
\begin{eqnarray*}
\lefteqn{\frac{1}{2}\frac{d}{dt}(\|q\|^2-\delta^2\|p\|^2)}\\
&\leq&
-\|q\|_1^2+L_R\|q\|\|q\|_\alpha+L_R\|p\|\|q\|_\alpha+\delta^2
L_R\|p\|^2+\delta^2 L_R\|p\|\|q\|\\
&\leq &
-\|q\|_1^2+L_R(\|p\|+\|q\|)\|q\|_1+\delta^2
L_R\|p\|^2+\delta^2 L_R\|p\|\|q\| \\
&\leq &
- \|q\|_1^2+ L_R(1+\tfrac1\delta) \|q\| \|q\|_1+ (1+\delta)L_R\|q\|^2 \\
&\leq &
- \tfrac12 \|q\|_1^2+  \Big( \tfrac12L_R^2(1+\tfrac1\delta)^2 + (1+\delta)L_R \Big) \|q\|^2 \\
\end{eqnarray*}
where we used Young inequality in the last step.

By (\ref{e:condition2}) we obtain using Poincare inequality
$$
\frac{d}{dt}(\|q\|^2-\delta^2\|p\|^2) \leq - \tfrac12 \lambda_* \|q\|^2 \;,
$$
which yields the desired cone invariance.

For the second claim consider now that if $p+q$ is outside the cone
at time $t_0$ ( i.e., $\|q(t_0)\|>\delta \|p(t_0)\|$)\,. Then by the
first result we have $\|q(t)\|>\delta \|p(t)\|$ for $t\in [0, t_0]$.
Then by (\ref{e:q}) we derive
\begin{eqnarray*}
 \frac{1}{2}\frac{d}{dt}\|q\|^2
&\leq&-\|q\|^2_1 +(\nu+z)\|q\|^2 +L_R \|q\|\|q\|_1+L_R\|p\|\|q\|_1\\
&\leq&-\|q\|^2_1 +(\nu+z)\|q\|^2 +L_R(1+\tfrac1\delta) \|q\|\|q\|_1\\
&\leq&-\tfrac12\|q\|^2_1 +\Big(\nu+z  +\tfrac12L_R^2(1+\tfrac1\delta)^2 \Big)\|q\|^2\\
\end{eqnarray*}
By (\ref{e:condition3})
$$\frac{d}{dt}\|q(t, \omega)\|^2 \leq (-\tfrac12\lambda_*+2z(\theta_t\omega)) \|q(t, \omega)\|^2\;.
$$
Then a  comparison principle  yields for almost all $\omega$
$$\|q(t, \omega)\|^2
\le \|q(0,\omega)\|^2\exp\left\{-\tfrac12\lambda_*t+2\int_0^t
z(\theta_{\tau}\omega)d\tau \right\}\,.
$$
Finally, $\|q(0,\omega)\|=\|v_s(0,\omega)-\bar v_s(0,\omega)\|\leq
\|u_0-\bar u_0\|e^{z(\omega)}$ yields the result.
\end{proof}

To finish the proof of Theorem \ref{thm:exRIM}, we also need the
following lemma.
\begin{lemma}\label{lem:backward sol}
For any given $T>0$, the following initial value problem
\begin{eqnarray*}
\dot{v}_c &=& \nu v_c+z v_c+
P_ce^{-z}B^{(R)}((v_c+v_s)e^z),
\;\; v_c(T)=\xi\in H_c \\
\dot{v}_s&=&(-L_s+\nu+z)v_s+P_se^{-z}B^{(R)}((v_c+v_s)e^z),\;\;
v_s(0)=h^s(v_c(0))
\end{eqnarray*}
has a unique solution $(v_c(t, \omega), v_s(t, \omega))\in C(0, T;
H_c\times H_s)$ which lies on the manifold
$\mathcal{M}^R_{cut}(\theta_t\omega)$ a.a..
\end{lemma}
\begin{proof}
This proof is same as Lemma 3.3 and Lemma 3.8 of \cite{DLSch03}
except that one should be careful of the nonlinearity is not
Lipschitz on space $H$ in proving the contraction property of the
solving operator. This step is same as the proof of contraction
property of operator $\mathcal{T}$ defined in (\ref{def:T}).
\end{proof}

Now following the contradicting discussion  in \cite{WD07} we verify
the following attracting property of the random invariant manifold.

\begin{theorem}\label{thm:reduce}
 Assume (\ref{e:condition1}),
(\ref{e:condition2}) and  (\ref{e:condition3}).
 For any solution $u(t, \omega)$ of cut-off
system (\ref{e:cut-off}), there is one orbit $U(t,\omega)$ on
$\mathcal{M}_{cut}^R(\theta_t\omega)$ with $P_cU(t,\omega)$ solves
the following equation
 $$
 \partial_t u_c=\nu u_c+P_cB^{(R)}(u_c+e^{z(\theta_t\omega)}h(\theta_t\omega, e^{-z(\theta_t\omega)}u_c))+\sigma u_c\circ\dot{W}(t)
$$
such that
$$ \|u(t, \omega)-U(t, \omega)\|\leq
D(t, \omega)\|u(0,\omega)-U(0,\omega)\|e^{-\lambda_*t}
 $$
where $D(t, \omega)$  is a tempered increasing  process defined by
 \begin{equation} \label{e:D}
D(t, \omega)= e^{z(\theta_t\omega)+\int_0^t
z(\theta_\tau\omega)\,d\tau}.
 \end{equation}
\end{theorem}
Then Theorem \ref{thm:exRIM} is a direct result of Theorem
\ref{thm:reduce}. Moreover we have the following result.
\begin{corollary} \label{cor:exp}
There is a LRIM $\mathcal{M}(\omega)$ for system (\ref{e:main}) in a
small ball $B(0, R)$ with $R$ and $\nu$ satisfy
(\ref{e:condition1}), (\ref{e:condition2}), and
(\ref{e:condition3}).  Moreover  for $\|u_0\|<R$
 $$
\text{dist}(\varphi(t, \omega)u_0 , {\cal M}(\omega)) \le
2RD(t,\omega)e^{-\lambda_*t}
 $$
for all $t\in[0,\tau_0)$ where
$\tau_0=\inf\{t>0\ :\ \varphi(t, \omega)u_0\not\in B_r(0)\}$.
\end{corollary}

%
\section{Shape of the Invariant Manifold}\label{sec:shapeIM}
%

In this section we show that the LRIM  $\mathcal{M}^R_{cut}$\,, see
(\ref{e:RIM}), which denotes the random invariant manifold of
equation (\ref{e:cut-off}) is locally  quadratic. That is the
Lipshitz-function $h$ (cf. Theorem \ref{thm:fixpoint}) describing
the manifold is quadratic near the fixed point  $0$. Suppose
$\sigma$ and $\nu$ are small and choose the cut-off $R$ also small.
Consider first  the LRIM of the transformed equation (\ref{e:rand
eq}). We saw in the previous Section that it is given by
 $$
\tilde{\mathcal{M}}^R_{cut}(\omega)=\text{Graph}(h)
=\{\left(\xi,h(\omega,\xi)\right):\ \xi \in H_c\}\,,
 $$
where
$$
h(\omega,\xi)
= P_s v^*(0,\omega,\xi)
=v^*_s(0,\omega,\xi)
$$
with $v^*$ the unique  fixed point of the
contraction $\mathcal{T}$ in the space $C^-_\eta$
(i.e., $v^*=\mathcal{T}(v^*)$, cf. Theorem \ref{thm:fixpoint}).
Thus for any $\xi\in H_c$ and $t\le0$
\begin{eqnarray}\label{e:vstar}
v^*(t)&=& e^{\nu t+\int_0^t z(s)ds}\xi
+\int_0^te^{\nu(t-\tau)+\int_\tau^t z(r)dr}e^{-z(\tau)}B^{(R)}_c(e^{z(\tau)}v^*(\tau)) d\tau \nonumber\\
&&+\int^t_{-\infty}e^{(-L_s+\nu)(t-\tau)+\int_\tau^tz(r)dr}
e^{-z(\tau)}B^{(R)}_s(e^{z(\tau)}v^*(\tau))d\tau\,.
 \end{eqnarray}
It was also established in the previous section, that
the local invariant manifold for the original equation  (\ref{e:main})
is given by (cf. (\ref{e:LRIM}))
 $$
 \mathcal{M}^R(\omega)=\{(\xi,e^{z(\omega)}h(\omega,e^{-z(\omega)}\xi)):\
\xi \in H_c\}\cap B_R(0)\,.
 $$
In order to describe the shape of the LRIM,
 we first establish two bounds on $v^*$ in $C^-_\eta$.
\begin{lemma}
\label{lem:som1}
Under Assumptions $A_1-A_3$ suppose $\lambda_* > \nu + \eta > 0$.
Then there is a constant $C>0$ such that for all  $\xi\in H_c$
$$
\|v^*\|_{C^-_\eta} \le C \|\xi\|
\quad\text{and}\quad
\|v^*_s\|_{C^-_\eta} \le C R \|\xi\|\;.
$$
\end{lemma}
\begin{proof}
For the first bound note that $\mathcal{T}(0)=e^{\nu t+\int_0^t z(s)ds}\xi$.
Thus
$$
\|\mathcal{T}(0)\|_{C^-_\eta}
= \sup_{t\le0} e^{(\nu+\eta)t}\|\xi\|
= \|\xi\|\;.
$$
Now  as $\mathcal{T}$ is a contraction
$$
 \|v^*\|_{C^-_\eta}
\le  \|\mathcal{T}(v^*)-\mathcal{T}(0)\|_{C^-_\eta}+ \|\mathcal{T}(0)\|_{C^-_\eta}
\le  \text{Lip}(\mathcal{T})  \|v^*\|_{C^-_\eta} + \|\xi\| \;.
$$
We obtain the first claim with $C=1/(1- \text{Lip}(\mathcal{T}) ))$.

For the second claim   we bound (\ref{e:vstar}) by using the
semigroup estimate from (\ref{e:SG}) together with the fact that
\begin{equation}\label{e:BRs}
\|B^{(R)}_s (v e^z) \|_{-\alpha} \le  C R\|v\|e^z\,,
\end{equation}
for $v \in H$, which follows immediately from Definition
\ref{def:cutoff} and Assumption A$_3$\,. Thus
\begin{eqnarray}
\lefteqn{\|v^*_s(t) \| = \|P_sv^*(t) \|}  \label{e:bvs} \\
&=&
\Big\| \int^t_{-\infty} e^{(-L_s+\nu)(t-\tau)+\int_\tau^tz(r)dr}
e^{-z(\tau)}B^{(R)}_s(e^{z(\tau)}v^*(\tau))d\tau\Big\|
\nonumber \\
&\le &
C \int^t_{-\infty}  (t-\tau)^{-\alpha}e^{-(\lambda-\nu) (t-\tau) +\int_\tau^tz(r)dr}
 R \| v^*(\tau)\| d\tau
\nonumber \\
&\le &
C R e^{\int_0^tz(r)dr}
\int^t_{-\infty}  (t-\tau)^{-\alpha}e^{-(\lambda-\nu) (t-\tau) } e^{-\eta\tau} d\tau   \| v^*\|_{C^-_\eta}\;.
\nonumber
\end{eqnarray}
Then using the first claim yields
$$
 \|v^*_s\|_{C^-_\eta}
\le C R \| \xi \|  \int^0_{-\infty}
|\tau|^{-\alpha}e^{-(\lambda-\nu-\eta) |\tau| }d\tau\,.
$$
Choosing $\lambda>\nu+\eta$ finishes the proof.
\end{proof}
Now we prove the first reduction step, which removes
the explicit dependence of $h$ on $v_s^*$.
Define  for $t\le0$ the cut-off
$\chi_R^*(t) = \chi_R(v^*(t)e^{z(t)})$
and the approximation
\begin{equation}
\label{e:defg1}
 g_1^*(t)=\int^t_{-\infty} e^{(-L_s+\nu)(t-\tau)+\int_\tau^tz(r)dr}
\chi_R^*(\tau)B_s(v^*_c(\tau),v^*_c(\tau))e^{z(\tau)}d\tau\,.
\end{equation}
\begin{lemma}\label{lem:som2}
Under Assumptions $A_1-A_3$ suppose  $\lambda_*>\eta+\nu>0$, then
$$
\|v_s^*-g_1^*\|_{C^-_\eta}
\le C R^2 \|\xi\|
\quad \text{ for all }
\xi\in H_c\;.
$$
\end{lemma}
\begin{proof}
We use
 \begin{eqnarray*}
e^{-z} B_s^{(R)}(e^z v^*) &=& \chi_R^*(t)  B_s(v^*,v^*) e^z\\
&=&  \chi_R^*(t) [ B_s(v^*_c,v^*_ce^z)+ 2B_s(v^*_s,v^*_c e^z)+B_s(v^*_s,v^*_se^z)].
 \end{eqnarray*}
Now from
$$
\| \chi_R^* v_c^* e^z \| \le R
\quad \text{and} \quad
\| \chi_R^* v_s^* e^z \| \le R
$$
we obtain
$$
\| e^{-z} B_s^{(R)}(e^z v^*) -  \chi_R^*(t) B_s(v^*_c,v^*_ce^z)\|_{-\alpha}
\le  CR\| v^*_s \|.
$$
 Analogously to (\ref{e:bvs}) we derive with $\lambda\in(\eta+\nu,\lambda_*)$ that
 \begin{eqnarray*}
\|v_s^*(t)-g_1^*(t)\|
&\le &
C \int^t_{-\infty}  (t-\tau)^{-\alpha}e^{-(\lambda-\nu) (t-\tau) +\int_\tau^tz(r)dr}
R \| v^*_s(\tau)\| d\tau\\
&\le &
C \int^0_{-\infty}  |\tau|^{-\alpha}e^{-(\lambda-\nu-\eta) |\tau|} d\tau\cdot
e^{\int_0^t z(r)dr} e^{-\eta t} R \| v^*_s\|_{C^-_\eta}.
\end{eqnarray*}
Thus using Lemma \ref{lem:som1}, the claim follows.
\end{proof}
Recall that  in order to bound $h$, we only need to consider
$v_s^*(0)$. From (\ref{e:defg1})
$$
g_1^*(0) =\int^0_{-\infty} e^{L_s\tau} e^{-\nu\tau}e^{\int_\tau^0
z(r)dr}
\chi_R^*(\tau)B_s(v^*_c(\tau),v^*_c(\tau))e^{z(\tau)}d\tau\,,
$$
where
$$
 v_c^*(\tau)e^{-\nu\tau} e^{\int_\tau^0 z(r)dr}- \xi
= \int_0^\tau e^{-\nu r+ \int_r^0  z(t)dt} e^{z(r)} \chi_R^*(r)
B_c(v^*(r),v^*(r)) dr\,.
$$
Using $B_c(v_c^*,v_c^*)=0$ and thus
$$
e^{z} \chi_R^* B_c(v^*,v^*)
=2B_c(v_s^*,e^z \chi_R^*v_c^*)+B_c(v_s^*,e^z \chi_R^*v_s^*)
$$
we obtain  for $\tau\le0$
 \begin{eqnarray}
 \lefteqn{\Big\| v_c^*(\tau)e^{-\nu\tau} e^{\int_\tau^0 z(r)dr}- \xi \Big\|}
 \label{e:bvcxi}\\
&\le & \int_\tau^0  e^{-\nu r+ \int_r^0  z(t)dt}R \|v^*_s(r)\| dr
\nonumber\\
&\le & R \int_\tau^0  e^{-(\nu+\eta) r}  dr \|v^*_s\|_{C^-_\eta} \le
CR^2 \|\xi\| e^{-(\nu+\eta)\tau}. \nonumber
 \end{eqnarray}
Hence,
$$
\Big\| v_c^*(\tau)- \xi e^{\nu\tau} e^{\int^\tau_0 z(r)dr} \Big\|
 \le CR^2 \|\xi\|e^{-\eta\tau} e^{\int^\tau_0 z(r)dr}\,.
$$
Thus for
\begin{equation} \label{e:defg2}
g_2^* :=  \int^0_{-\infty} e^{L_s\tau}
\chi_R^*(\tau)B_s(\xi, v^*_c(\tau))e^{z(\tau)}d\tau
\end{equation}
we obtain using (\ref{e:bvcxi}) and estimates analogous to (\ref{e:bvs})
 \begin{eqnarray*}
 \|g_1^*(0)-g_2^* \|
&\le& C \int^0_{-\infty} |\tau|^{-\alpha} e^{-\lambda|\tau|}
R \cdot \|  e^{-\nu\tau}e^{\int_\tau^0 z(r)dr}v^*_c(\tau)-\xi\| d\tau\\
&\le& C R^3\|\xi\|   \int^0_{-\infty} |\tau|^{-\alpha}
e^{-\lambda|\tau|} e^{-(\nu+\eta)\tau}  d\tau\,.
 \end{eqnarray*}
Thus together with Lemma \ref{lem:som2} we have
 \begin{theorem}
Under Assumptions $A_1-A_3$ suppose   $R\le 1$,
and $\lambda_*>\nu+\eta>0$. Then
$$
\|h(\omega, \xi)-g_2^*(\omega,\xi)\| \le C R^2 \|\xi\|
 \quad\text{for all}\quad
\xi\in H_c\,.
$$
 \end{theorem}
This theorem is how far we can get without $\omega$-dependent
deterministic bounds. In the following we will rely on the following
lemma (cf. (\ref{e:ou})).

 \begin{lemma} \label{lem:bouz}
 There is a random variable $K(\omega)$
such that $K(\omega)-1$  has a standard exponential distribution
and for  all $t\le0$
 $$
\int_0^t z(r)dr+z(t)=z(0)+\sigma\omega(t)  \le  \sigma
(K(\omega)+|t|)\,.
 $$
 \end{lemma}
 \begin{proof}
Define
\begin{equation}
\label{def:supWiener}
\tilde{K}(\omega)=\sup_{s\le0}\{\omega(s)+s\}\,.
\end{equation}
For example from \cite{Yor97} we know that
$2\tilde{K}$ has a standard exponential distribution.
Obviously,
$$
\sigma\omega(t) \le \sigma (\tilde{K}(\omega)+|t|) \quad\text{for
all}\quad t\le0\,.
$$
Moreover, by definition
\begin{eqnarray} \label{e:z0}
 z(0)&=& \sigma \int_{-\infty}^0 e^s\omega(s)ds\\
&\le&\sigma \int_{-\infty}^0 e^sds\tilde{K}(\omega) - \sigma
\int_{-\infty}^0 se^s ds =\sigma(\tilde{K}(\omega)+1)\,. \nonumber
\end{eqnarray}
 \end{proof}

\begin{remark}
\label{rem:K+-}
 For $s<0$ we have
$$
|\omega(s)| \le \max\{\omega(s),-\omega(s)\} \le
K^\pm(\omega)+|s|\,,
$$
where
$$ K^\pm(\omega)=K(\omega)+K(-\omega)$$
and  $K(-\omega)$ has the same law as $K(\omega)$.
Furthermore, for $|z(0)|$ a similar estimate is true.
\end{remark}

Define
\begin{equation}
\label{e:defg3} g_3^*:=  \int^0_{-\infty} e^{L_s\tau} \chi_R^*(\tau)
e^{z(\tau)}e^{\nu\tau}e^{\int_0^\tau z(r)dr}  d\tau  B_s(\xi,
\xi)\,.
\end{equation}
Using (\ref{e:bvcxi}) yields
\begin{eqnarray*}
 \|g_2^* -g_3^* \|
&\le& C \int^0_{-\infty} |\tau|^{-\alpha} e^{-\lambda|\tau|} e^{z(\tau)}
 \| \xi \| \cdot         \|v_c^*(\tau) - \xi e^{\nu\tau}e^{\int^\tau_0 z(r)dr}\| d\tau \\
&\le& C R^2\|\xi\|^2    \int^0_{-\infty} |\tau|^{-\alpha}
e^{-\lambda|\tau|} e^{\int_0^\tau z(r)dr}e^{z(\tau)}e^{-\eta\tau}
d\tau\,.
\end{eqnarray*}
Then by Lemma \ref{lem:bouz} we obtain  for
$\sigma<\frac12(\lambda-\eta)$
\begin{eqnarray}
 \|g_2^* -g_3^*\|
&\le& C R^2\|\xi\|^2    \int^0_{-\infty} |\tau|^{-\alpha}
 e^{(\eta-\lambda+2\sigma) |\tau|}  e^{2\sigma K(\omega)}  d\tau \nonumber\\
&\le&C R^2\|\xi\|^2 e^{2\sigma K(\omega)} \,. \label{e:cut1}
\end{eqnarray}
Now we eliminate the cut-off.
Suppose $\|\xi\|e^{z(0)}\le R$,
then we can use the smoothness of $\chi_R$ (mean  value theorem)
to obtain
\begin{eqnarray*}
|\chi_R^*(\tau)-1|
&=&|\chi_R(v^*(\tau)e^{z(\tau)})-\chi_R(\xi e^{z(0)})| \\
&\leq & \tfrac{C}{R}  \|v^*_s(\tau)e^{z(\tau)}\|
+ \tfrac{C}{R}\|v^*_c(\tau)e^{z(\tau)}-\xi e^{z(0)}\|\\
&\leq & C(1+R)\|\xi\| e^{-\eta\tau+z(\tau)+\int_0^\tau z(r)dr}
+\frac{C}{R}\|\xi\| \cdot |e^{z(0)}-e^{\nu\tau +\int_0^\tau z(r)dr} |
\end{eqnarray*}
where we used Lemma \ref{lem:som1} and (\ref{e:bvcxi}).

Now from Lemma \ref{lem:bouz}, and the fact that $z(t)+\int_0^t
z(r)dr = z(0)+\sigma \omega(t)$ we derive
$$ |\chi_R^*(\tau)-1|
\le   C(1+R)\|\xi\|e^{-\eta\tau+\sigma|\tau|+\sigma K(\omega)} +
C|1-e^{\nu\tau+\sigma\omega(\tau)}|e^{z(0)}\,.
$$
Thus
\begin{equation}
 \label{e:remcut}
|\chi_R^*(\tau)-1|
\le C(\|\xi\|(1+R)+\sigma K_2(\omega))e^{-\eta\tau}e^{\sigma|\tau|+\sigma K(\omega)}\;,
\end{equation}
where we define
\begin{equation}
\label{e:defK2} K_2(\omega):=\sup_{\tau \le 0}\Big| \frac{1-e^{\nu
\tau+\sigma\omega(\tau)}}{\sigma e^{\eta|\tau|}}\Big|\,.
\end{equation}
For this random constant $K_2$ we obtain the following relation to
$K(\omega)$. Note that this result is by far not optimal, but it is
rather simple to derive.
\begin{lemma} \label{lem:K2}
Suppose $|\nu|+|\sigma|<\tfrac{\eta}{2}$ and $|\nu|\le|\sigma|$, then
there is a constant depending only on $\eta$ such that
$$
K_2(\omega) \le  Ce^{|\sigma|K^\pm(\omega)} (1+K^\pm(\omega))\,.
$$

\end{lemma}
\begin{proof}
We use the rather crude estimate $|1-e^x| \le |x|e^{|x|}$ for all $x\in\mathbb{R}$.
Recall Remark \ref{rem:K+-} to obtain
\begin{eqnarray*}
 K_2(\omega)&\le&
\sup_{\tau \le 0} \frac{|\nu||\tau|+|\sigma||\omega|}{|\sigma|} e^{|\nu| |\tau|+|\sigma\omega(\tau)|- \eta|\tau|}\\
&\le & \sup_{\tau \le 0} \frac{(|\nu|+|\sigma|)|\tau|+|\sigma|}{|\sigma|K^\pm(\omega)}
 e^{|\sigma|K^\pm(\omega)} e^{ |\tau|(|\nu|+|\sigma|- \eta)}\\
&\le & 2 e^{|\sigma|K^\pm(\omega)}  \sup_{\tau \le
0}(|\tau|+K^\pm(\omega))e^{-\tfrac12|\tau|\eta}\,.
\end{eqnarray*}
\end{proof}
Now we proceed to bound $g_3^*$
\begin{eqnarray}
\lefteqn{
 \left\| g_3^*  -  \int^0_{-\infty} e^{L_s\tau}e^{z(\tau)}e^{\nu\tau}e^{\int_0^\tau z(r)dr}  d\tau  B_s(\xi, \xi) \right\| }\nonumber\\
&\le&
C   \int^0_{-\infty} |\tau|^{-\alpha}e^{(-\lambda+2\sigma) |\tau|}  e^{\sigma K(\omega)}
|1-\chi_R^*(\tau)|  d\tau\|\xi\|^2 \nonumber\\
&\stackrel{(\ref{e:remcut})}{\le}&
C   \int^0_{-\infty} |\tau|^{-\alpha}e^{(-\lambda+\eta+2\sigma) |\tau|}  e^{2\sigma K(\omega)}d\tau
(\|\xi\|(1+R)+\sigma K_2(\omega))\|\xi\|^2 \nonumber\\
&\le&
C   e^{2\sigma K(\omega)}  (\|\xi\|+\sigma K_2(\omega))\|\xi\|^2
\label{e:cut2}
\end{eqnarray}
for $R\le1$ and $\eta-\lambda+2\sigma<0$.

For the final step note that
$$
 \int^0_{-\infty} e^{L_s\tau}
e^{z(\tau)}e^{\nu\tau}e^{\int_0^\tau z(r)dr}  d\tau  = e^{z(0)}
\int^0_{-\infty} e^{L_s\tau} e^{\nu\tau}e^{\sigma\omega(\tau)}
d\tau\,.
$$
We finally bound
\begin{eqnarray}
\lefteqn{ \left\|  \int^0_{-\infty} e^{L_s\tau}
e^{\nu\tau}e^{\sigma\omega(\tau)}  d\tau B_s(\xi,\xi) - L_s^{-1} B_s(\xi,\xi)\right\|}\nonumber\\
&=& \left\| \int^0_{-\infty} e^{L_s\tau}(1-
e^{\nu\tau}e^{\sigma\omega(\tau)})  d\tau B_s(\xi,\xi)\right\|\nonumber\\
&\le& C   \int^0_{-\infty} |\tau|^{-\alpha}e^{(-\lambda+\eta) |\tau|} \sigma K_2(\omega)d\tau\|\xi\|^2 \nonumber\\
&\le& C \sigma K_2(\omega) \|\xi\|^2
\label{e:cut3}
\end{eqnarray}
for $\lambda>\eta$.

Combining (\ref{e:cut1}),  (\ref{e:cut2}), and  (\ref{e:cut3}),
we proved the following theorem:
\begin{theorem}
\label{thm:shape2}
Let Assumptions $A_1-A_3$ be true.
Suppose  $\lambda_*>\eta+2\sigma$, $\lambda_*>\nu+\eta$, $\sigma>0$,
 and $\|\xi\|e^{z(0)} \le R \le1$, and let $h$ be the fixed point given by
Theorem \ref{thm:fixpoint}.
Then
$$
\|h(\omega, \xi)-e^{z(0)} L_s^{-1}B_s(\xi,\xi)\| \le C (e^{2\sigma
K(\omega)}+1)\cdot(\|\xi\|+R^2+\sigma K_2(\omega))\cdot\|\xi\|^2\,,
$$
where $K$ and  $K_2$ are defined in (\ref{def:supWiener}) and (\ref{e:defK2}).
\end{theorem}

This is the final main Theorem of this section.
The main Theorem \ref{thm:shape} is now a simple corollary.

\begin{proof}[Proof of Theorem \ref{thm:shape}]
Define
$$
\Omega_K=\Big\{\omega\in\Omega_0\ :\ K^\pm(\omega) >
1/\sqrt{\sigma}\Big\}\,.
$$
By Lemma \ref{lem:bouz}
this set has probability less that  $C\exp\{-1/\sqrt{\sigma}\}$.
Moreover, on the complement $\Omega_K^c$ we have
$$
K\le \tfrac{1}{\sqrt{\sigma}}
\quad \text{and} \quad K_2\le  \tfrac{1}{\sqrt{\sigma}}
$$
and (\ref{e:claim_shape}) follows from Theorem \ref{thm:shape2}, as
similar to Lemma \ref{lem:K2} we also have  on $\Omega_K^c$
$$
e^{z(0)}\|\xi\| \le e^{C\sqrt\sigma}\|\xi\| \le R\,.
$$
 \end{proof}

\section*{Acknowledgments}
Supported by Bosch-foundation Nr. 32.5.8003.0010.0 (``Angewandte
Mathematik") and NSFC No. 10701072. Part of the work was done during
a stay of W. Wang at the University of Augsburg.

\end{document}